\documentclass[12pt]{amsart}

\usepackage{amsfonts}
\usepackage{amssymb}

\theoremstyle{plain}
\newtheorem{theorem}{Theorem}[section]
\newtheorem{proposition}[theorem]{Proposition}

\newtheorem{lemma}[theorem]{Lemma}

\theoremstyle{definition}

\newtheorem{example}[theorem]{Example}

\newcommand{\bN}{\mathbb{N}}

\begin{document}

\baselineskip 6.5mm

\title[Sequences of bounds for the spectral radius]{Sequences of bounds for the spectral radius of a positive operator
\footnote{Linear Algebra and its Applications 574 (2019), 40--45} }

\author{Roman Drnov\v sek}


\begin{abstract}
In 1992, Szyld provided a sequence of lower bounds for the spectral radius of a nonnegative matrix $A$,
 based on the geometric symmetrization of powers of $A$. 
In 1998, Ta\c{s}\c{c}i and Kirkland proved a companion result by giving a sequence of upper bounds
for the spectral radius of $A$, based on the arithmetic symmetrization of powers of $A$.
In this note, we extend both results to positive operators on $L^2$-spaces.
\end {abstract}

\maketitle

\noindent
 {\it Key words}:   numerical radius,  spectral radius, positive operators, kernel operators \\
 {\it Math. Subj. Classification (2010)}:   47A12,  47A10, 47B34, 47B60 \\

\section{Introduction}
Let $A = (a_{i j})_{i, j =1}^n$ be a nonnegative matrix, i.e., $a_{i, j} \ge 0$ for all $i$ and $j$. 
In the literature, much attention has been paid to provide 
upper and lower bounds for the spectral radius $r(A)$ of $A$. 
Here $r(A)$ is defined as $\max\{|\lambda_1|, \ldots, |\lambda_n|\}$, 
where $\{\lambda_i\}_{i=1}^n$ are the eigenvalues of $A$. This number is an eigenvalue of $A$, and it is
called the Perron root of $A$.
Szyld \cite{Sz92} gave an increasing sequence of lower bounds for $r(A)$ 
that are based on the geometric symmetrization of powers of $A$. 
More precisely, let $S(A)$ be the matrix whose $(i,j)$ entry is equal to $\sqrt{a_{i j} a_{j i}}$,
and define $\rho_k = r(S(A^{2^k}))^{2^{-k}}$. It is shown in  \cite{Sz92} that 
$\rho_0 \le \rho_1 \le \ldots \le \rho_k \le  r(A)$ for all $k$. Easy examples show that the sequence 
$\{\rho_k\}_{k \in \bN}$ does not converge to $r(A)$ in general; if we take e.g. 
$$ A = \left( \begin{matrix}
0 & 1 & 0 \cr
0 & 0 & 1 \cr
1 & 0 & 0 
\end{matrix} \right) , $$
then $\rho_0 = \rho_1 = \ldots = 0$, while $r(A) = 1$.
On the other hand, Ta\c{s}\c{c}i and Kirkland \cite{TK98} provided  a decreasing sequence of upper bounds for $r(A)$
that are based on the arithmetic symmetrization of powers of $A$. Specifically, let $M(A) = (A + A^T)/2$ and define 
$\sigma_k = r(M(A^{2^k}))^{2^{-k}}$. It is proved in \cite{TK98} that 
$\sigma_0 \ge \sigma_1 \ge  \ldots \ge \sigma_k \ge r(A) $ for all $k$ and that 
 the sequence $\{\sigma_k\}_{k \in \bN}$ converges to $r(A)$. 
In this paper we extend both results to positive operators on $L^2$-spaces.
We should mention that  in \cite{Dr03} the inequality $\sigma_0 = r(M(A)) \ge r(A)$ was already extended to this setting.

Throughout the note, let $\mu$ be a $\sigma$-finite positive measure on a set $X$. 
We consider bounded (linear) operators on the complex Hilbert space $L^2(X, \mu)$. 
The norm in $L^2(X, \mu)$ is denoted by $\| \cdot \|_2$. 
An operator $A$ on  $L^2(X, \mu)$ is said to be {\it positive} 
if it maps nonnegative functions to nonnegative ones. Given operators $A$ and $B$ on $L^2(X, \mu)$,
we write $A \ge B$ if the operator $A-B$ is positive.
The operator norm and the spectral radius of an operator are denoted by $\| \cdot \|$ and 
$r(\cdot)$, respectively. The {\it numerical radius} of an operator $A$ on $L^2(X, \mu)$ is defined by 
$$ w(A) := \sup \{ | \langle A f, f \rangle | : f \in L^2(X, \mu), \| f \|_2 = 1 \} . $$
If, in addition, $A$ is positive, then we have 
$$ w(A) = \sup \{ \langle A f, f \rangle  : f \in L^2(X, \mu) , f \ge 0,  \| f \|_2 = 1 \} . $$
Indeed, this follows from the estimate
$$  | \langle A f, f \rangle | \le \int_X \! |A f| \, |f| \, d\mu \le 
    \langle A |f|, |f| \rangle  $$
that holds for any  $f \in L^2(X, \mu)$. It is well-known \cite{GR97} that 
$$ r(A) \le w(A) \le \|A\| $$
for all bounded operators $A$ on $L^2(X, \mu)$. If, in particular, $A$ is selfadjoint, then we have
$r(A)=w(A)=\|A\|$.

Let $A$ be a positive  operator on $L^2(X,\mu)$. The \textit{arithmetic symmetrization} $M(A)$ of $A$ is 
the positive selfadjoint operator on $L^2(X,\mu)$ defined by $M(A) = (A + A^*)/2$. 
Since $\langle M(A) f, f \rangle = \langle A f, f \rangle$ for any nonnegative function $f \in L^2(X, \mu)$, we have
\begin{equation}
 w(M(A)) = w(A) \ . 
\label{wM_A}
\end{equation}

Let $K$ be a positive kernel operator on $L^2(X,\mu)$ with a kernel $k$, that is, $k: X \times X \rightarrow [0,\infty)$ is a measurable function such that $(Kf)(x) = \int_X k(x, y) f(y) d\mu(y)$ for all $f \in L^2(X, \mu)$ and for almost all $x \in X$.
The \textit{geometric symmetrization} $S(K)$ of $K$ is the 
positive selfadjoint kernel operator on $L^2(X,\mu)$ with the kernel equal to $\sqrt{k(x,y)k(y,x)}$
at a point $(x, y) \in X \times X$. Note that $S(K)$ is well-defined on the whole $L^2(X,\mu)$, because 
the kernel of $S(K)$ is smaller than or equal to the kernel of $M(K)$
by the inequality of arithmetic and geometric means. 
It was proved in \cite[Proposition 2.7]{DP06} that 
\begin{equation}
 r(S(K)) \le r(K) . 
\label{rS_K}
\end{equation}
Note that our results do not generalize beyond $L^2$-spaces.

\section{Results}

We begin with an observation that seems to be new also in the finite-dimensional case.

\begin{lemma}
\label{S_square}
If $K$ is a positive kernel operator on $L^2(X, \mu)$, then 
$$ S(K^2) \ge S(K)^2 . $$
\end{lemma}

\begin{proof}
Let us compare the kernels of both operators. 
If $k$ is the kernel of $K$, then the kernel of $K^2$ at a point $(x,y) \in X \times X$ is equal to
$\int_X k(x,z) k(z,y) d\mu(z)$, and so the kernel of $S(K^2)$ at a point $(x,y)$ equals 
$$ \sqrt{\left( \int_X k(x,z) k(z,y) d\mu(z) \right) \left(\int_X k(y,z) k(z,x) 
d\mu(z) \right)} . $$
On the other hand, the kernel of $S(K)^2$ at a point $(x,y)$ is equal to
$$ \int_X \sqrt{k(x,z) k(z,x)} \sqrt{k(z,y) k(y,z)} d\mu(z) \ . $$
Now the desired inequality is proved by an application of the Cauchy-Schwarz inequality.
\end{proof}

We now extend the finite-dimensional result due to Szyld \cite[Theorem 2.2]{Sz92}.

\begin{theorem}
\label{lower}
Let $K$ be a positive kernel operator on $L^2(X, \mu)$, and let $\rho_n = r(S(K^{2^n}))^{2^{-n}}$, 
$n \in \bN \cup \{0\}$. Then, for each $n$, 
$$ \rho_0 \le \rho_1 \le \ldots \le \rho_n \le  r(K) . $$
\end{theorem}

\begin{proof}
By \eqref{rS_K}, we have $r(S(K^{2^n})) \le r(K^{2^n}) =  r(K)^{2^n}$, which implies that 
$\rho_n \le r(K)$ for all $n \in \bN \cup \{0\}$. To finish the proof, it is enough to show that 
$$
 r(S(K^2)) \ge r(S(K))^2 . 
$$ 
By Lemma \ref{S_square}, we have $S(K^2)) \ge S(K)^2$ which implies easily that 
$r(S(K^2)) \ge r(S(K)^2) = r(S(K))^2$ as desired.

\end{proof}

The following theorem is an infinite-dimensional generalization of \cite[Theorems 1 and 2]{TK98}.

\begin{theorem}
\label{upper}
Let $A$ be a positive operator on $L^2(X, \mu)$, and let $\sigma_n = r(M(A^{2^n}))^{2^{-n}}$, $n \in \bN \cup \{0\}$.  Then, for each $n$, 
$$ \sigma_0 \ge \sigma_1 \ge  \ldots \ge \sigma_n \ge r(A) . $$
Furthermore, the sequence $\{\sigma_n\}_{n \in \bN}$ converges to $r(A)$.
\end{theorem}

\begin{proof}
By \eqref{wM_A}, we have 
$$ r(M(A^{2^n})) = w(M(A^{2^n})) = w(A^{2^n}) \ge r(A^{2^n})=r(A)^{2^n} , $$
and so $\sigma_n \ge r(A)$ for all $n$. 

To show that  the sequence $\{\sigma_n\}_{n \in \bN}$ is decreasing, it is enough to see that 
$r(M(A^2)) \le r(M(A))^2 = r(M(A)^2)$, or equivalently $w(M(A^2)) \le w(M(A)^2)$.
It follows from the inequality $\|A f - A^*f\|^2_2  \ge 0$ that we have 
$$  \langle A^* A f, f \rangle + \langle A A^* f, f \rangle \ge 2 \langle A^2 f, f \rangle $$
for all nonnegative functions $f \in L^2(X, \mu)$. This implies that 
$$  \langle (A+ A^*)^2 f, f \rangle \ge 4 \langle A^2 f, f \rangle, $$
and so we obtain the desired inequality $w(M(A)^2) \ge w(A^2) = w(M(A^2))$, where we have also used \eqref{wM_A}.

For each $n \in \bN \cup \{0\}$ we have 
$$ r(M(A^{2^n}))=  w(M(A^{2^n})) =  w(A^{2^n}) \le  \|A^{2^n}\| , $$ 
and so
$$ r(A) \le \sigma_n \le  \|A^{2^n}\|^{2^{-n}} . $$ 
Since $\|A^{2^n}\|^{2^{-n}} \rightarrow r(A)$ as $n \rightarrow \infty$, 
the sequence $\{\sigma_n\}_{n \in \bN}$ converges to $r(A)$.
This completes the proof.
\end{proof}

Examples  in \cite{TK98} explain Theorem \ref{upper} in the finite-dimensional case. The following example 
further illustrates it in the infinite-dimensional setting.

\begin{example}
Let $A$ be a weighted unilateral shift on $l^2$ with weights $\{1, 4, 1, 4, 1, 4, \ldots\}$, that is, the operator 
defined by $A(x_1, x_2, \ldots) = (0, x_1, 4 x_2, x_3, 4 x_4, \ldots)$. Then 
$M(A) x =\frac{1}{2} (x_2, x_1+4 x_3, 4 x_2+x_4, x_3+ 4 x_5, 4 x_4+ x_6, \ldots)$.
It is not difficult to verify that $r(M(A))= \|M(A)\| = 5/2$, and so $\sigma_0 = 5/2$.
Furthermore, $A^2 x = (0, 0, 4 x_1, 4 x_2, 4 x_3, \ldots) = 4 S^2 x$, where $S$ is the unilateral shift on $l^2$.
Therefore, $r(A)^2 = r(A^2) = 4 r(S)^2 = 4$, and so $r(A) = 2$. Since $M(A^2) = 2\, (S^2+(S^*)^2)$, we have 
$r(M(A^2))= \|M(A^2)\| = 4$, so that $\sigma_1 = 2$. 
Similarly, we obtain that $\sigma_n = 2$ for all $n=2, 3, 4, \ldots$.  

More generally, let $p$ be a positive integer, and let $A_p$ be a weighted unilateral shift on $l^2$ with weights 
$\{\underbrace{1, 1, \ldots,1}_{2^p-1}, 2^{2^p}, \underbrace{1, 1, \ldots,1}_{2^p-1}, 2^{2^p}, \ldots\}$, so that 
$A_1 = A$. Then $A_p^{2^p} = 2^{2^p} S^{2^p}$, and
$r(A_p)^{2^p} = r(A_p^{2^p}) = 2^{2^p} r(S)^{2^p} = 2^{2^p} $, so that $r(A_p) = 2$.
Since $M(A_p^{2^p}) = 2^{2^p-1}\, (S^{2^p}+(S^*)^{2^p})$, we have 
$r(M(A_p^{2^p}))= \|M(A_p^{2^p})\| = 2^{2^p}$, so that $\sigma_p= 2$. 
Then $\sigma_n  = 2$ for all $n \ge p$ by Theorem \ref{upper}.  
One can also show that $\sigma_0 = r(M(A_p))= \|M(A_p)\| \ge 2^{2^p-1}$ and that 
$\sigma_n  > 2$ for all $n < p$. 
\end{example}

We complete this note with an application of the inequality  \eqref{rS_K}. 
Several authors have studied the spectrum and the spectral radius of selfadjoint kernel operators. 
If $K$ is a positive kernel operator on $L^2(X, \mu)$,
then $S(K)$ is a selfadjoint kernel operator. Therefore, if we can compute $r(S(K))$, then  the inequality  \eqref{rS_K}
provides the lower bound for $r(K)$. Let us illustrate this with the following proposition.

\begin{proposition}
Let $g : [0,1] \times [0,1] \rightarrow [0,\infty)$ be a measurable function such that, for some $M \ge 1$, 
$\frac{1}{M} \le g(x, y) \le M$ and $g(x,y) g(y,x) = 1$ for all $x$ and $y$. Then 
the function $k(x,y) = \min\{x,y\} \cdot g(x, y)$ is the kernel of the positive kernel operator $K$ on $L^2([0,1])$, and
we have  
$$ r(K) \ge \frac{4}{\pi^2} \ . $$ 
\end{proposition}

\begin{proof}
Since $k(x,y)$ is a bounded nonnegative function, it defines  the positive kernel operator $K$ on $L^2([0,1])$.
The kernel of $S(K)$ at a point $(x, y) \in [0,1] \times [0,1]$ is equal to $\min\{x,y\}$. 
It follows from \cite[Exercise 6.5.3, p.103 and p.271]{EMT} that $r(S(K)) =  4/\pi^2$. 
Now, the proof is finished with an application of the inequality  \eqref{rS_K}. 
\end{proof}

\vspace{3mm}
{\it Acknowledgment.} The author acknowledges the financial support from the Slovenian Research Agency  (research core funding No. P1-0222).

\vspace{2mm}

\baselineskip 6mm
\noindent
Roman Drnov\v sek \\
Department of Mathematics \\
Faculty of Mathematics and Physics \\
University of Ljubljana \\
Jadranska 19 \\
SI-1000 Ljubljana, Slovenia \\
e-mail : roman.drnovsek@fmf.uni-lj.si 

\end{document}